\documentclass[psamsfont,11pt]{amsart}

\usepackage{amsmath,amssymb,
}
\everymath{\displaystyle}

\newtheorem{theorem}{Theorem}[section]
\newtheorem{defn}[theorem]{Definition}
\newtheorem{cor}[theorem]{Corollary}
\newtheorem{ex}[theorem]{Example}

\newtheorem{lem}[theorem]{Lemma}
\newtheorem{prop}[theorem]{Proposition}
\theoremstyle{definition}
\newtheorem{rem}[theorem]{Remark}

\def\Ham{\mathrm{Ham}} 

\def\ov{\overline}  \def\wt{\widetilde} 

\def\Z{\mathbb{Z}}  \def\R{\mathbb{R}} \def\C{\mathbb{C}}   \def\J{\mathcal{J}} \def\P{\mathcal{P}} \def\H{\mathcal{H}}  \def\B{\mathcal{B}}     \def\H{\mathcal{H}} \def\L{\mathcal{L}} \def\C{\mathbb{C}} \def\E{\mathcal{E}}  \def\M{\mathcal{M}}    

  \def\om{\omega}

\def\delbar{\ov{\partial}} \def\bs{\backslash}

\def\hra{\hookrightarrow}  \def\la{\langle} \def\ra{\rangle} \def\LA{\left\langle} \def\RA{\right\rangle} \def\sra{\rightarrow} \def\lra{\longrightarrow}  
\def\codim{\text{codim}}  \def\T{\mathbb{T}}

\begin{document}

\title{Lagrangian Circle actions.}

\author{Cl\'ement Hyvrier}
\address{
C\'egep Saint-Laurent\\
D\'epartement de Math\'ematiques\\
625 avenue Sainte-Croix\\
Montreal, QC, H4L 3X7\\
Canada}
\email{
chyvrier@cegep-st-laurent.qc.ca
}

\maketitle


%

\begin{abstract} We consider paths of Hamiltonian diffeomorphism preserving a given compact monotone Lagrangian in a symplectic manifold that extend to an $S^1$--Hamiltonian action. We compute the leading term of the associated Lagrangian Seidel element. We show that such paths minimize the Lagrangian Hofer length. Finally we apply these computations to Lagrangian uniruledness and to  give a  nice presentation of the Quantum cohomology of real lagrangians in Fano symplectic toric manifolds.
\end{abstract}

\section{Introduction}  Let  $(M^{2n},\om)$  denote a symplectic manifold and let $L$ be a compact connected Lagrangian in $M$. Here, we will consider  exact Lagrangian loops of $L$. 
 Consider the  set of Hamiltonian  isotopies starting at the identity and with ending point a Hamiltonian diffeomorphism preserving $L$:
 $$\P_L\Ham(M,\om):=\{\gamma:[0,1]\stackrel{C^{\infty}}{\lra} \Ham(M,\om)|\gamma_0=id,\quad \gamma_1(L)=L\}.$$ 
These are  the paths  generating  exact Lagrangian loops of $L$ (see M.Akveld and D.Salamon \cite{AkveldSalamon}).  
 Similarly to loops of Hamiltonian diffeomorphisms, such paths define automorphisms of the Lagrangian quantum homology of $L$ when defined (see Hu-Lalonde-Leclerq \cite{HuLalondeLeclercq}).  Any such automorphism can be seen as multiplication by an invertible element of the lagrangian quantum homology called \emph{Lagrangian Seidel element}. 
     
 For weakly exact lagrangians it has been shown in \cite{HuLalondeLeclercq} that  the Lagrangian Seidel morphism is always trivial, hence the Seidel element is  simply given by the fundamental class of $L$, when defined.

 In this paper we are interested in computing Lagrangian Seidel elements  for those paths admitting extensions to a loop of Hamiltonian diffeomorphisms coming from $S^1$--action on $(M,\om)$.  In other words for the elements in $\P_L\Ham(M,\om)$ that are homotopic to paths which, when squared, yield an $S^1$--Hamiltonian action on the symplectic manifold. 
  
To ensure that all the automorphisms we want to compute are well-defined we will assume that $(M,L)$ is \emph{monotone}. If $\mu_L:\pi_2(M,L)\sra \Z$ denotes the Maslov index and if $I_{\om}:\pi_2(M,L)\sra \R$ is the $\om$--valuation then monotonicity means that 
\[\begin{cases} I_{\om}=\lambda \mu_L &\text{ $\lambda>0$}; \\
N_L:=\inf_{ \pi_2(M,L)}\{\mu_L(A)|A\neq 0\}\geq 2 &.\end{cases}  \]

In this framework we will compute the leading  term of the Seidel element.  We will also show that in some cases all the other terms vanish. For instance, this is the case for the (monotone) totally real lagrangians in toric manifolds. These computations can be seen  as the relative counter-part of the computation done  by McDuff and Tolmans  for the Seidel elements of an $S^1$--Hamiltonian action on $M$  \cite{McDuffTolman}.

These calculations imply that such Hamiltonian paths cannot define null-homotopic exact Lagrangian loops.  We will further show that for paths giving  $S^1$--Hamiltonian actions when squared,  the Lagrangian Hofer length is minimized, hence they define relative geodesics in their homotopy class with fixed endpoints. This is not that surprising considering that  the  obtained Hamiltonian loops define a geodesics in their homotopy class as shown by D.McDuff and J.Slimowitz \cite{McDuffSlimowitz}.  We point out that such results can be useful to study Lagrangian uniruling defined by P.Biran and O.Cornea \cite{BiranCorneaUniruling}. The main class of examples for  which we concretely apply the calculations mentioned above are the \emph{real Lagrangians} in Fano symplectic  toric manifolds, that is a symplectic manifold $(M^{2n},\om)$ with a Hamiltonian action of $\T^n$ with some positivity assumption. These Lagrangians are the fixed points set of the unique anti-symplectic involution preserving the moment map of the torus-action. Under a monotonicity assumption,  L.Haug \cite{Haug} showed  that these Lagrangian submanifolds are wide with respect to  $\Z_2$--Laurent polynomials coefficient ring. This means that  the corresponding Lagrangian Quantum homology must split as a product of the $\Z_2$--Morse homology of $L$ with the coefficient ring.  We will show that the multiplicative quantum relations of $L$ are generated by Lagrangian Seidel elements. This can be seen as a relative version of the observation made by D.McDuff and S.Tolman in \cite{McDuffTolman}.  Using L.Haug's result we then provide a description of the Lagrangian Quantum homology as quotient of some polynomial ring exactly analogous to that given in the absolute case.

\subsubsection{Formulation of the  main result}  We need to introduce Lagrangian Quantum homology. Roughly speaking this is the  homology theory obtained by  deforming  the Morse differential on $L$  taking into account pseudo-holomorphic disks in $M$ with boundary in $L$. More precisely, this is the homology of the  \emph{pearl complex} 
$$C(L; f,g_L;J;\Lambda_L):=(R\la Crit f\ra\otimes \Lambda_L,d_Q),$$
where $(f,g_L)$ is a Morse-Smale pair for $L$,   $\Lambda_L:=R[q^{-1},q]$ is the ring of  $R$--Laurent polynomials graded by requiring that $|q|=1$ and where  the differential $d_Q$ can be written a sum of $\Lambda_L$-linear maps
$$d_Q=d_0+d_1 \otimes q^{-N_L}+d_2\otimes q^{-2N_L}+...$$
where $d_0$ stands for the Morse differential of $f$ and
\begin{equation*} d_k: R\la Crit_r(f)\ra \sra R\la Crit_{r+kN_L-1}(f) \ra
\end{equation*} 
is obtained by counting \emph{pearl trajectories}, i.e. chains of gradient flow lines of $f$ and $J$-holomorphic disks in $M$ with boundary on $L$ with cumulative Maslov index $kN_L$. In the present text we will only be considering $R=\Z_2$ as ground coefficients. 
From P.Biran and O.Cornea \cite{BiranCorneaQuantumHomology}  the homology of this complex is generically well defined under the monotonicity assumption.   We will denote by $QH(L;\Lambda_L)$ the corresponding Lagrangian quantum homology. For more on this, we refer to  \cite{BiranCorneaQuantumHomology} and the references therein.

 Lagrangian Seidel elements are invertibles $S_L(\gamma)\in QH(L;\Lambda_L)$ where $\gamma\in\P_L\Ham(M,\om)$. Their definition involves counting pearl trajectories with pseudo-holomorphic sections  in the Hamiltonian fibration, $M \hra P_{\gamma}\stackrel{\pi}{\sra} D^2$,   associated with $\gamma$, which boundary lies on  the Lagrangian  $\pi^{-1}(\partial D^2)$ (see Section 3).  Here, we will discuss, to some extent,  what happens when $\gamma$ is a  path of Hamiltonian diffeomorphisms of $L$ admitting an $S^1$--Hamiltonian action extension:
 
\begin{defn}\label{definitionextension} We say that  $\gamma\in\P_L\Ham(M,\om)$ extends to an $S^1$--Hamiltonian action if  it is homotopic relative endpoints to a path  $\gamma':[0,1]\sra \Ham(M,\om)$ such that the concatenation of $\gamma'$ with itself,  $(\gamma')^2$,  defines an  $S^1$--Hamiltonian action. We denote by $\P$ the set of such paths.
\end{defn}

In particular, if $H_t:M\sra \R$ denotes the family of (normalized) Hamiltonian functions generating $\gamma'$,    the action of $(\gamma')^2$ is generated by a smooth time-independant (normalized) function $K:M\sra\R$.  Thus the normalized Hamiltonian generating $\gamma'$   is also time-independant and one has $K=2H$.   In the remaining of the article, unless otherwise mentioned,  we will   assume  that $\gamma$ is   $\gamma'$, i.e. already extends to an $S^1$--action.   Let $F_{max}$  denote the maximal fixed point set component of the $S^1$--action associated to $(\gamma')^2$.  Throughout the paper we will restrict our attention to the following case:

\emph{\begin{itemize}\label{hypotheseaction}
\item[(A1):]  $F_{max}$ is \emph{semifree}, i.e. the action is semifree in a neighbourhood of  $F_{max}$. 
\item[(A2):]  $L$ intersects $F_{max}$  cleanly, and the intersection  $F^L_{max}:=L\cap F_{max}$  is a Lagrangian submanifold of $F_{max}$. \end{itemize}}
 
\begin{rem}
\begin{itemize}
\item[1)] Note that (A2) implies $\dim(F^L_{max})=1/2\dim(F_{max})$.
\item[2)] If  the gradient flow of $K$ is contained in $L$, then  \emph{A2} holds. Indeed,  let $J$ be  an $S^1$--invariant almost complex structure on $M$ compatible with $\om$. At any $x\in F_{max}$ we have the splitting:
$$T_xM=T_xF_{max}\oplus N_x\cong \ker (1-d\gamma(x))\oplus N_x$$
where $N_x$ is the symplectic complement of $T_xF_{max}=\ker (1-d\gamma(x))$. Since $J$ is $S^1$--invariant, $J$ is also split and so is the Hermitian metric induced $g_J:=\om(\cdot,J\cdot)$. By the assumption, we have that $T_xL$ is compatible with that splitting meaning that:
$$T_xL=T_x F^L_{max}\oplus_{g_J} \{v\in T_xL| v\in N_x\}.$$
Since $L$ is Lagrangian, both terms in the above summand are maximally isotropic subspaces of $T_xF_{max}$ and $N_x$ respectively, and so the claim follows.  
\end{itemize}
\end{rem}
%

For a fixed point $x$, let $w(x)$ the sum of the weights at $x$. We recall that for an $S^1$--invariant $\om$--compatible almost complex structure on $M$, the action of $S^1$ on $T_xM\cong \C^n$ is conjugate to a product of circle actions $z\mapsto e^{2\pi k_j t}z$, $t\in S^1$. Then the sum of the weights $w(x)=\sum k_j$ defines a locally constant function and as such only depends on the connected component of the fixed point set in which $x$ lies.
We will denote by  $w_{max}$  the sum of the weights for points in the fixed point set $F_{max}$.  
 The main result of this paper  is the following:

\begin{theorem}\label{calculelementdeSeidel} Let $L$ be a monotone compact Lagrangian submanifold of $(M,\om)$.  Let  $\gamma\in\P$ satisfying the assumptions (A1) and (A2). The corresponding Lagrangian Seidel element is given by 
\begin{eqnarray*}S_L(\gamma)=[F^L_{max}]\otimes q^{-w_{max}}+\sum_{\{B\in\pi_2(M,L)|\mu_L(B)>0\}}a_B\otimes q^{-w_{max}-\mu_L(B)}
\end{eqnarray*}
where $\deg(a_B)=\dim(F^L_{max})+\mu_L(B)$. 
In particular, if $\codim(F_{max})=2$ then all lower order terms $a_B$ vanish. \end{theorem}

As an example let us mention the case of  half of  a Hamiltonian loop fixing a given divisor $D$ (a facet of the moment polytope) in a Fano symplectic toric manifold. The endpoints of such path fix the real Lagrangian in this manifold (e.g consider a meridian $S^1$ in $S^2$ and the action of rotating around the poles) and  assumptions $(A1)$ and $(A2)$ are verified. Hence, if the real Lagrangian is monotone, one concludes that the corresponding Lagrangian Seidel element is given by $[D\cap L]\otimes q$.  It is worth noticing that lower order terms  may appear in situations that are reminiscent of those exposed in \cite[Theorem 1.10]{McDuffTolman} (see Example \ref{Theexample} in Section 5). In fact the more general results obtained in \cite{McDuffTolman} cannot go through with the techniques used here, as regularity of symmetric almost complex structures generally fails. \\

 Theorem \ref{calculelementdeSeidel} can be applied to deduce some results about uniruledness of Lagrangian submanifolds. As defined by O.Cornea and P.Biran in \cite{BiranCorneaUniruling},  a monotone Lagrangian manifold $L$ in $M$ is said to be \emph{1--uniruled}, or \emph{uniruled}, if there exists a second category Baire subset of families of almost complex structures of $M$ with the property that: for each such almost complex structure there is a non-constant pseudo-holomorphic disk in $M$ with boundary on $L$ passing through any generic point of $L$.

\begin{theorem}\label{Lagrangianuniruledness} Let $L\subset M$ be a closed monotone Lagrangian and suppose there is $\gamma\in \P$ such that the corresponding $S^1$--Hamiltonian action verifies hypothesis (A1) and (A2), then $L$ is uniruled. 
 \end{theorem}

Theorem \ref{calculelementdeSeidel}  also implies that  exact Lagrangian loops respecting the assumptions of the Theorem cannot be null-homotopic since the higher order term is not $[L]$. We further  show the following:

\begin{theorem}\label{lengthminimizing} Let $L\subset M$ be a closed monotone Lagrangian. Suppose there exists $\gamma\in \P_L\Ham(M,\om)$  such that $\gamma^2$ defines a semi-free $S^1$--Hamiltonian action on $M$. Then $\gamma$ minimizes the Hofer length in its homotopy class with fixed endpoints. 
\end{theorem}

The paper is organized as follows. In Section 2 we recall the definition of  Hamiltonian fibrations associated to loops or paths  of Hamitlonian diffeomorphisms. In Section 3 we introduce  Lagrangian Seidel elements. Section 4 is devoted to the proof of Theorems \ref{calculelementdeSeidel} and \ref{lengthminimizing}. In Section 5 we apply our results to Lagrangian uniruledness (we give the proof of Theorem \ref{Lagrangianuniruledness}) and  we show that the multiplicative relations for the Lagrangian Quantum Homology of real Lagrangians in Fano symplectic toric manifolds  are generated by Lagrangian Seidel elements (more precisely we prove Proposition \ref{Quantumpresentation}).\\

 \noindent\textbf{Aknowledgement.} I would like to thank  Fran\c{c}ois Charette, Octav Cornea, Tobias Ekholm and Yasha Savelyev for usefull discussions.  I would also like to thank Fran\c{c}ois Charette for his suggestions and comments on an earlier version of this note, that  helped improve the presentation of the paper; in particular for explaining to me how to simplify the assumptions in  Theorem \ref{Lagrangianuniruledness}.

\section{Hamiltonian fibrations} A Hamiltonian fibration $\pi:P\sra B$  with fiber $(M^{2n},\om)$, and compact symplectic base $(B,\om_B)$, is a symplectic fibrations which structure group reduces to $\Ham(M,\om)$. Such fibrations are naturally equipped with a family $\{\om_b\}_{b\in B}$ of symplectic forms in the fibers    $\pi^{-1}(b)$ induced by $\om$. It was shown by Guillemin, Lerman, and Sternberg  \cite{GLS}, and by McDuff and Salamon in full generality,   that Hamiltonian fibrations are symplectically trivial over the 1-skeleton of $B$ and that they admit an Erhesman connection on $TP$ which holonomy around any loop is Hamiltonian. This latter condition can be formally expressed as follows: there exists  a closed 2--form $\tau\in \Omega^2(P)$ extending $\omega$. The corresponding horizontal distribution, i.e. a direct complement  in $TP$ to the \emph{vertical subbundle} $Vert:=\ker d\pi$, is given by 
$$Hor_{\tau}(p):=\{w\in T_pP |\tau(w,v)=0\quad\forall v\in Vert_p=\ker d\pi(p)\}.$$
Different connection forms $\tau$ as above may determine the same horizontal distribution. However, a unique choice can be made by requiring that   $\pi_*\tau^{n+1}=0$, where $\pi_*$ denotes integration over the fibers. When this latter normalization condition is satisfied we say that $\tau$ is a \emph{coupling form}. 
Note that such Hamiltonian fibrations   admit symplectic structures:
 $$\Omega_{c}:= \tau+c\pi^* \om_B$$
where $c$ is a large enough strictly positive real number.

Recall that an $\om$--tame almost complex structure $J$ on a symplectic manifold $(M,\om)$ is a smooth endomorphism of $TM$ such that $$\forall p\in M,\quad (J(p))^2=-id_{T_pM}\quad \text{and}\quad \omega(\cdot,J \cdot)>0.$$
Let $\mathcal{J}(M,\omega)$ denote the set of $\om$--tame almost complex structures, which is contractible and non-empty \cite[Chapter 2]{MS2}. The symplectic manifold $(P,\Omega_{c})$ admits  $\Omega_{c}$--tame almost complex complex structures $J_P$ that are \emph{compatible with $\pi$ and $\tau$ or fibered} in the following sense: 
\begin{itemize}
  \item $d\pi \circ J_P=J_B\circ d\pi$, where  $J_B\in\mathcal{J}(B,\omega_B)$
  \item $J_b:=\left.J_P\right|_{\pi^{-1}(b)}\in\mathcal{J}(\pi^{-1}(b),\omega_{b})$ for all $b\in B$,
  \item $J_P$ preserves the horizontal distribution induced by $\tau$.
\end{itemize}
Let $\mathcal{J}(P,\Omega_{c},\tau,\pi)$ denote the set of such almost complex structures. In fact, for fixed $\tau$,   any family $J=\{J_b\}_{b\in B}$  of $\omega_{b}$-tame almost complex structures and any given $J_B$ give rise to a unique fibered $J_P\in \mathcal{J}(P,\Omega_{c},\tau,\pi)$ for some $c\in\R$.

Hamiltonian fibrations as above are given two canonical cohomology classes. The first one is the \emph{vertical first Chern class} induced by any family of almost complex structures $\{J_b\}_{b\in B}$   and defined by 
$$c_v:=c_1(Vert)\in H^2(P,\Z).$$
The second one  is the   deRham cohomology class of the coupling form $[\tau]\in H^2(P,\R)$: this is the unique class such that 
$$\iota^*[\tau]=[\om]\quad\text{and}\quad [\tau]^{n+1}=0.$$
 In what follows we shall only consider Hamiltonian fibrations over $D^2$ and $S^2$.

\subsection{Hamiltonian fibrations associated to a loop of Hamiltonian diffeomorphisms} Let $\gamma_t\in \L \Ham(M,\om)$. Such a loop defines a Hamiltonian fibration over $S^2$ via the \emph{clutching construction}. Namely, let $D^+$ and $D^-$  denote the unit discs in $\C$, but  with opposite orientations. Then set 
$$P_{\gamma}:=\left.D^{+}\times M\sqcup D^{-}\times M\right/ (e^{i2\pi t},x)\sim (e^{i2\pi t},\gamma_t(x)),\,\,t\in [0,1].$$
This is obviously a Hamiltonian fibration over $(S^2,\om_{FS})$ where  $\om_{FS}$ denotes the Fubini-Study form on $S^2$ with total area two. Its isomorphism class only depends  on the homotopy class of the 
loop $\{\gamma_t\}$. In fact  any Hamiltonian fibration over $S^2$ can be obtained in this way \cite{LMP}.  If $\tau_{\gamma}$  denotes the corresponding  coupling form, then  for big enough positive constant $c$
    $$\Omega_{c}:=\tau_{\gamma}+c\pi^* \om_{FS},$$
is symplectic.  

When  $\{\gamma_t\}$ is given by an $S^1$--action on $F$, $P_{\gamma}$ can be described in the following way. Let $p:S^3\sra S^2$ denote the Hopf fibration. The product $S^3\times M$ admits the free $S^1$ action:
$$e^{it}.((z_0,z_1),x)\mapsto((e^{-it}z_0,e^{-it}z_1), \gamma_t(x)),$$
and the quotient $S^3\times_{S^1} M$ can be identified with $P_{\gamma}$. In that setting, the coupling form is obtained by considering a connection 1--form on $S^3$. Namely, let $\alpha\in \Omega^1(S^3)$ be the standard contact structure, so that $d\alpha=p^*\om_{FS}$ where $\om_{FS}$ is the Fubini-Study form on $S^2$ normalized to have area 1. Then $\om-d(K\alpha)\in\Omega^2(S^3\times M)$ defines a closed basic form hence defines aclosed 2-form on the quotient: 
$$\tau_{\gamma}=pr_*(\om-d(K\alpha))\in \Omega^2(S^3\times_{S^1} M)$$
where $pr:S^3\times M\sra S^3\times_{S^1} F$ denotes the projection and $pr_*$ denotes integration over the fibers of $pr$. This form clearly extends $\om$ on the fiber $F$, and since $K$ is normalized the integral $\pi_*\tau_{\gamma}^{n+1}$ vanishes. Thus, $\tau_{\gamma}$ is a coupling form.

In this particular framework, each fixed point of the action yields a section of $P_{\gamma}$. Namely, for $x\in Fix$, the corresponding section is $$\sigma_x:=S^3\times_{S^1}\{x\}.$$
The following will be useful later on:
\begin{lem}\cite[Lemma 2.2]{McDuffTolman}\label{McDuffTolman1}  If $x$ is a fixed point of the Hamiltonian circle action of $\gamma$, then 
$$c_v(\sigma_x)=w(x)\quad and \quad \tau_{\gamma}(\sigma_x)=-K(x)$$
Moreover, if $B$ is the class of a sphere    formed by  the $\gamma$-orbit of an arc between $x$ and $y$, then $B=\sigma_x-\sigma_y$. 
\end{lem}

Consider now an $S^1$--invariant $\om$--tame almost complex structure on $M$. Note that the standard complex structure $J_0$ in $\C^2$ is  $S^1$--invariant. Its restriction to $S^3$ preserves the contact structure $\ker \alpha$. Let $R$ denote the Reeb vector field associated to $\alpha$, and $X_K$ denote the Hamiltonian vector field of the $S^1$--action on $M$. 
 Then  any vector field $v=[v_1,v_2]\in T(S^3\times_{S^1}M)$ admits a unique representative in $T(S^3\times M)$  lying in $\ker \alpha\oplus TM$ given by 
$$(v_1-\alpha(v_1)R, v_2+\alpha(v_1) X_K).$$ 
It follows that $J_0\times J$ descends to a well-defined almost complex structure $\ov{J}$ on $P_{\gamma}$, which is  obviously fibered and tames $\Omega_{c}$ for $c>\max K$. Note that $\sigma_x$ is then $\ov{J}$--holomorphic.

\subsection{Hamiltonian fibrations associated to $\gamma\in \P_L\Ham(M,\om)$}
To any path $\gamma \in \P_L\Ham(M,\om)$ one  associates a Hamiltonian fibration over the 2-disc  as follows. Let $\mathbb{H}$ denote the upper half plane in $\C$, and set 
$$D^+_{+}:=\{z=x+iy\in \mathbb{H}| |z|\leq 1\}$$
and $D^+_{-}$ is the same half disc but with opposite orientation. In particular, the compactified upper-half plane   $\ov{\mathbb{H}}=D^+_{+}\cup_{\varphi} D^+_{-}$  coincides with the disc of radius one $D^2$ (here,
 $\varphi: D^{+}_+\sra D^+_-$ is given by $z\sra \bar{z}^{-1}$). 

The Hamiltonian fibration associated to the path is then given by the \emph{half clutching construction}:
$$P_{\gamma}:=\left.D^+_+\times M\sqcup D^+_-\times M\right/ (e^{i\pi t},x)\sim (e^{i\pi t},\gamma_t(x)).$$
Again, its isomorphism class only depends on the homotopy class with fixed endpoints of $\gamma_t$. This manifold carries symplectic structures such that the subbundle defined by collecting all the copies of $L$ along the boundary of $\ov{\mathbb{H}}$ is a  Lagrangian submanifold $N$ fibering over $S^1$. If $\gamma_t$ denote a Hamiltonian isotopy representing $\gamma\in \Ham_L(M,\om)$, then 
$$N\cong \bigcup_{t\in[0,1]}\{e^{2\pi i t}\}\times \gamma_t(L)$$
This Lagrangian actually embeds in $P_{\gamma}\cong D^2\times M$ and is called exact Lagrangian loop (see \cite{AkveldSalamon}, \cite{HuLalondeLeclercq}). Considering the specific case where $\gamma\in \P$ one easily sees that $N$ is a Lagrangian. 

A connection form on $P_{\gamma}$ is  given by 
\begin{equation}\label{explicitcouplingform}\tau_{\gamma}=\om-dH\wedge dx-dH\wedge dy-\frac{H}{\pi} dx\wedge dy,
\end{equation}
and symplectic structures are explicitly given by
	$$\Omega_{c}= \tau_{\gamma}+ \frac{c}{\pi} dx\wedge dy,\quad \text{for $c>0$ big  enough} .$$
 One verifies that the parallel transport of $\tau_{\gamma}$ preserves the fiber bundle $N$. 
 Letting   $L_s$ be the copy of $L$ lying in the fiber over $s\in \partial D^2$, this is the same as saying that the vector field along $N$ 
$$(s_0,\left.\frac{d}{ds}\right|_{s=0}\gamma_s(p)),\quad p\in L_{s_0}$$
is horizontal with respect to $\tau_{\gamma}$.   Hence,  $N$ is Lagrangian submanifold of $P_{\gamma}$ for the symplectic forms $\Omega_{\kappa}$. 
Note also that $\tau$ vanishes on $N$.  We will denote by $\mathcal{T}(\gamma)$ the set of connection 2--forms on $P_{\gamma}$ which parallel transport along the boundary  preserves $N$. Equivalently these are the connection 2--forms that vanish identically on $N$ (see \cite{AkveldSalamon}, Lemma 3.1.).   The set of relative cohomology classes $[\tau]\in H^2(D^2\times M, N;\Z)$ associated to elements $\tau\in \mathcal{T}(\gamma)$  is a 1--dimensional affine space: for any  $\tau_0,\tau_1\in \mathcal{T}(\gamma)$
 \begin{equation*} [\tau_1]-[\tau_0]=c(\tau_1,\tau_0) [\frac{1}{\pi}dx\wedge dy],\quad s(\tau_1,\tau_0)\in\R.
 \end{equation*}
Hence, for $c$ big enough $\tau=\tau_0+\frac{c}{\pi} dx\wedge dy$ is symplectic, and if small enough $-\tau$ is symplectic. Set 
$\mathcal{T}^{\pm}(\gamma):=\{\tau\in \mathcal{T}(\gamma) | \pm\tau^{n+1}>0\}$. 
Then, any value of $c$ for which $\tau$ is non-symplectic lies between the following two real numbers
\begin{equation*}\epsilon^+(\tau_0, N):=\inf \{c(\tau,\tau_0)|\tau\in\mathcal{T}^{+}(\gamma)\}
\end{equation*}
and 
\begin{equation*}\epsilon^-(\tau_0, N):=\inf \{c(\tau,\tau_0)|\tau\in\mathcal{T}^-(\gamma)\}.
\end{equation*}
The width of the corresponding non-symplectic interval  $\epsilon (N)$ does not depend on the reference point  and is given by\begin{equation*}\epsilon (N)=\epsilon^+(\tau_0, N)-\epsilon^-(\tau_0, N) .
\end{equation*}

\subsection{Hofer length of exact Lagrangian loops} Let $\gamma_t$ denote a Hamiltonian isotopy representing $\gamma\in  \P_L\Ham(M,\om)$. As seen precedently this defines an exact Lagrangian loop $N\subset D^2\times M$. Assume $\gamma_t$ is generated by the family $H_t$ of  Hamiltonians. In the case where $M$ is not compact these must be compactly supported. When $M$ is closed then we assume that $H_t$ is normalized for all $t$.  The Hofer length of $N$ is defined to be
\begin{eqnarray*} \ell(N)=
\int_0^1 (\max_{x\in L_t} H_t(x)-\min_{x\in L_t} H_t(x)) dt,
\end{eqnarray*}
where $L_t=\gamma_t(L)$. Subsequently we will consider minimizing the Hofer length within the isotopy class of $\gamma$ with fixed endpoints. This is the same as minimizing $\ell(N)$ within its isotopy class of exact Lagrangian loops.    In other words we will examine the quantity:
\begin{equation*} \nu(N; M,\om ):=\inf_{N'} \ell(N')
\end{equation*}
where the infimum is taken over all exact Lagrangian loops that are Hamiltonian isotopic to $N$. The following theorem is due to M. Akveld and D. Salamon:

\begin{theorem}[M. Akveld and D. Salamon, \cite{AkveldSalamon}, Theorem B] \label{theoremeBAkveldSalamon} For every exact Lagrangian loop $N$
$$\epsilon(N)\leq \nu(N).$$ 
\end{theorem}

\subsection{The doubling procedure}
Let   $\gamma^2$ denote the   loop of Hamiltonian diffeomorphisms  associated to the path $\gamma\in\P$. First note that we have an obvious embedding
$$\iota_2:P_{\gamma}\hra P_{\gamma^2}.$$  
Taking the pull-back of $pr_*(\om-d(K\alpha))$ under $\iota_2$ actually yields $\tau_{\gamma}$ in \eqref{explicitcouplingform}.

In fact,  $P_{\gamma^2}$ is made of two copies of $P_{\gamma}$ glued together along their boundary. More precisely,
 $$ P_{\gamma^2}=P_{\gamma}\cup_{\varphi}P_{\gamma}$$
where
$$\varphi:\partial P_{\gamma}\sra \partial P_{\gamma},\quad (s, x)\to (-s, \gamma_1(x)).$$

This happens to be   useful subsequently  to induce information on $P_{\gamma}$ from $P_{\gamma^2}$. In particular, this will be handy when dealing with holomorphic sections. Concerning sections let us note that not only  any section in $P_{\gamma^2}$ gives rise to a section in $P_{\gamma}$   (which may not  have boundary on the Lagrangian), but also, any section $$\sigma :(D^2,S^1)\sra (P_{\gamma},N),\quad z\to (z,u(z))$$ can be doubled to give a section in $P_{\gamma^2}$ with the equator being constrained to the Lagrangian $\iota_2(N)$. The new section is given by
$$\sigma_{db}:S^2\sra P_{\gamma^2},\quad z\to \begin{cases} (z,u(z)) & \text{if $z\in \ov{\mathbb{H}}$}\\
(z,\gamma_1 (u(e^{-i\pi}z)))& \text{if $z\in e^{i\pi}\cdot\ov{\mathbb{H}}$}.
\end{cases}$$ 
Note that this is  well-defined and continuous.

\section{Lagrangian Seidel elements} 
Here we recall the definition of Seidel elements in both the absolute and relative cases. 

\subsection{Section classes and vertical Maslov index}
Consider the pair $(P_{\gamma},N)$  as above. We say that $B\in \pi_2(P_{\gamma},N)$ is a \emph{section class} if and only if $\pi_*(B)=[D^2,S^1]$ is the positive generator. We say that it is a \emph{fiber class} if $B$ lies in the image of the inclusion map $$\pi_2(M,L)\sra \pi_2(P_{\gamma},N)$$ thus implying $\pi_*B=0$. As shown in \cite{HuLalondeLeclercq}, the following sequence is exact in the middle:
\begin{equation}\label{exactsequencedisc} \pi_2(M,L)\sra \pi_2(P_{\gamma},N)\sra \pi_2(D^2,S^1)
\end{equation}

\begin{defn}(\cite{HuLalondeLeclercq}) Let $u:D^2\sra P_{\gamma}$ be  a smooth map representing $B\in\pi_2(P_{\gamma},N)$.The \emph{vertical Maslov Index} of $B$ is the Maslov Index of the pair $(u^*(T^vP_{\gamma}), u^*T^vN)$, where $T^vN$ denotes the vertical tangent bundle of the bundle $N$. We will denote this number by $\mu_{\gamma}^v(B)$ or $\mu_v$ for simplicity.
\end{defn}

 It is a well-defined $\Z$--valued  morphism of $\pi_2(P_{\gamma},N)$ which further verifies that: 
 $$\mu_N(B)=\mu_{\gamma}^v(B)+2\quad \text{and}\quad \mu_N(B-B')=\mu_L(B-B')$$ for two section classes $B$ and $B'$.  We make the following identification on  $\H_{rel}\subset H_2(P_{\gamma}, N)$ the set of section classes: 
 $$B\sim B'\Leftrightarrow \tau_{\gamma}(B-B')=0=\mu_{\gamma}^v(B-B').$$ 
This is obviously an equivalence relation. We set  $\widetilde{\H_{rel}}:= \left.\H_{rel}\right/\sim$ and we will denote by $[B]$ the equivalence class of $B\in \H_{rel}$. The following follows directly from the definitions of the vertical Maslov class and of the doubling of a section:

\begin{lem}\label{relationMaslovclassandChernclass} Let   $\gamma^2$ denote the   loop of Hamiltonian diffeomorphisms  associated to the path $\gamma\in\P$.  If $\sigma$ represents the class $[\sigma]\in H_2(P_{\gamma}, N)$ then $\sigma_{db}$ represents the class $[\sigma^2]:=[\sigma\#\sigma]\in H_2(P_{\gamma^2})$ and one has 
$$c_v([\sigma_{db}])=\mu_{\gamma}^v([\sigma])\quad \text{and}\quad \tau_{\gamma^2}([\sigma_{db}])=2\tau_{\gamma}([\sigma])$$
\end{lem}

\begin{proof} The second one follows by definition. For the first one:  
$$2c_v([\sigma^2])=\mu_v([\sigma^2])=\mu_v([\sigma])+\mu_v([\sigma])=2\mu_v([\sigma])$$
hence the claim. 
\end{proof}

\subsection{Holomorphic  and anti-holomorphic sections} Let $j$ denote the standard complex structure on the disc, that is the anti-clockwise rotation by 90 degrees on the plane. Let $\{J_z\}$, $z\in D^2$, denote a smooth family of $\om$--tamed almost complex structures in $M$. Let 
$$\mathbb{H}:=H\otimes dx+H \otimes dy.$$ 
This is a  1--form over $D^2$ with values in $C_0^{\infty}(M)$. Let $X_{\mathbb{H}}$ be the induced 1--form with values in Hamiltonian vector fields of $M$, and let $X^{0,1}_{\mathbb{H}}$ denote the corresponding $(j,J)$ anti-holomorphic part.
These data provide an almost complex structure on $P_{\gamma}$ as follows:
\[  J_P(\tau,J)(z,x):=\left(\begin{array}{cc} j(z) &  0 \\ 
X^{0,1}_{\mathbb{H}}(z,x) & J_z(x)\\
\end{array} \right)\]
It is easy to check that this is fibered. Furthermore, if $\tau \in \mathcal{T}^{\pm}(\gamma)$ then $J_P(\tau,\pm J)$ is $\pm\tau$--tamed. In fact $J_P(\tau, J)$ is $\Omega_c$--tamed for $c$ large enough.


We consider the following boundary value problem for smooth sections $u:D^2\sra P_{\gamma}$:
\begin{equation}\label{R-Hproblem}
J_P\circ du= du\circ j \quad \text{and} \quad u(\partial D^2)\subset N.
\end{equation}
Fix a section class $A$ and let  $\M(P_{\gamma},A;\tau, J)$ denote the moduli space of $J_P(\tau,J)$--holomorphic sections representing $A$:
\[ \M(P_{\gamma},A;\tau, J):=\{u:D^2\sra P_{\gamma}|\eqref{R-Hproblem}\,\, \text{and}\,\, [u]=A\}.\]
For generic $(\tau, J)$ this is a manifold of dimension $n+\mu^v(A)$ [\cite{AkveldSalamon}, \cite{HuLalonde}]. Taking $-J$ instead of $J$, the moduli space $\M(P_{\gamma},A;\tau, -J)$ is similarly defined and is generically a manifold of dimension  $n-\mu^v(A)$.

\begin{rem}\label{fixedpoint=holsection} Note that a fixed point $x\in L$ of $\gamma$ defines a section of $P_{\gamma}$: $u:D^2\sra P_{\gamma}$, $z\sra (z,x)$. This section is   $J_P(\tau,\pm J)$--holomorphic. Indeed, for $z=s+it$ and for $u=(z,\tilde{u})$, the first part of  equation \eqref{R-Hproblem} is equivalent to 
\[\frac{\partial \tilde{u}}{\partial s}+J_z(\tilde{u})\frac{\partial\tilde{u}}{\partial t}+X_H(\tilde{u})-J_z(\tilde{u})X_H(\tilde{u})=0\]
If $x\in L$ is fixed under $\gamma$, then $\tilde(u)=x$ so that $\frac{\partial \tilde{u}}{\partial s}=\frac{\partial \tilde{u}}{\partial t}=0$. Furthermore, $X_H(\tilde{u})=0$ since $x$ is a fixed point.
\end{rem}

\subsection{The relative Seidel element} We now define the  relative Seidel element associated to a path $\gamma\in \P_L\Ham(M,\om)$. First, we recall the definition of the \emph{Lagrangien Seidel morphism} given in \cite{HuLalondeLeclercq}. 
Consider a Morse-Smale pair $(F,G)$ where $F\in C^{\infty}(N)$ is a Morse function and $G$ is a metric on $N$ such that: 
\begin{itemize}
\item[1)] $f_{\pm 1}:=\left.F\right|_{\pm 1}$ are Morse functions on the fibers $L_{\pm 1}$ over $\pm 1$ of $N$;
\item[2)]  $Crit f_{+}\cup Crit f_-=Crit F$;
\item[3)] $\max f_-+1<\max f_+$. 
\item[4)]  there exists neighbourhoods $U_{\pm}\cong (-\epsilon,\epsilon)$ of $\pm 1\in S^1$  trivializing the fiber bundle $N$,  with respect to which $\left. F\right|_{U_{\pm}}(t,x)=f_{\pm}(x)\mp \varphi(t)$ for any $(t,x)\in U_{\pm}\times L_{\pm 1}$ and where $\varphi$ is quadratic of index 1 at 0. 
\item[5)] we also ask that $\left.G\right|_{U_{\pm}}$ is a product metric $dt^2+G_{\pm}$ and that  $( f_{\pm}, G_{\pm})$ are Morse-smale pairs.
 \end{itemize}
 Such pairs $(F,G)$ can be chosen generically. 
 
Fix $\tau\in\mathcal{T}(\gamma)$ and a family $J=\{J_z\}_{z\in D^2}$ of $\om$--tamed almost complex structures of $M$. Let  $J_P(\tau,J)$ be the corresponding fibered almost complex structure of $P_{\gamma}$.  For $[\sigma]\in \widetilde{\H_{rel}}$ and for $x_-\in Crit(f_-)$ and $x_+\in Crit(f_+)$ let
$$ \M^{pearl}(x_-,x_+, [\sigma]; \tau, J, F,G)$$
denote the set of pearl trajectories from $x_-$ to $x_+$ representing the equivalence class of section classes $[\sigma]$. 
In particular elements of this moduli space have one $J_P(\tau,J)$--holomorphic section component with boundary on $N$,  and possibly many  $J_{\pm 1}$--holomorphic disk components with boundary on $L_{\pm 1}$.

For simplicity we will omit the auxiliary data $\tau, J,F$ and $G$ in the notations. This set  is a manifold of dimension 
$$ \dim \M^{pearl}(x_-,x_+, [\sigma])=|x_-|_N-|x_+|_N+\mu_N([\sigma])-1=|x_-|_L-|x_+|_L+\mu_v([\sigma]).$$

The Lagrangian Seidel morphism is defined to be:

\begin{defn}[\cite{HuLalondeLeclercq}] For $\gamma\in \P_L\Ham(M,\om)$, the Lagrangian Seidel morphism associated to $\gamma$ is  an endomorphism
\begin{eqnarray*}S_L(\gamma): R\la Crit_{\star}(f_-)\ra &\sra & R\la Crit_{\star}(f_+) \ra \\
x_- &\to & \sum_{\{[\sigma]\in \widetilde{\H_{rel}}||x_+|_L=|x_-|+\mu_v([\sigma])\}}\#_{\Z_2}  \M^{pearl}(x_-,x_+, [\sigma])x_+q^{-\mu_v([\sigma])}
\end{eqnarray*}
\end{defn}
This  only depends on the homotopy class  of paths with fixed endpoints of $\gamma$. 
Moreover, since $L$ is monotone with $N_L\geq2$, the  Lagrangian  Seidel morphism is a chain morphism (with respect to the pearl differential) and is generically well-defined with respect to  the data of  $J_P$,  $F$ and $G$ [\cite{BiranCorneaUniruling},\cite{HuLalondeLeclercq}]. Since it is a chain morphism and since $[L_-]$, the maximum of $Crit(f_-)$, defines a pearl cycle, $S_L(\gamma)([L_-])$ is also a cycle. 

\begin{defn} For $\gamma\in \P_L\Ham(M,\om)$, the \emph{Lagrangian Seidel element} associated to $\gamma$ is  the homology class \begin{eqnarray*}[S_L(\gamma)([L_-])]:=\left[\sum_{\{[\sigma]\in \widetilde{\H_{rel}}||x_+|_L=n+\mu_v([\sigma])\}}\#_{\Z_2}  \M^{pearl}([L_-],x_+, [\sigma])x_+q^{-\mu_v([\sigma])}\right]\in QH_{n}(L_+).
\end{eqnarray*}
\end{defn}

\section{Proofs of Theorem \ref{calculelementdeSeidel} and Theorem \ref{lengthminimizing}}

\subsection{Proof of Theorem \ref{calculelementdeSeidel}}
For  the proof   we follow the steps given by D. McDuff and S. Tolman in \cite{McDuffTolman} where they compute the absolute Seidel element of $S^1$--Hamiltonian manifolds.  Let $\sigma_x$ denote the relative section class associated to a fixed point $x\in L$. In the special case where $x\in F_{max}$ we write $\sigma_{max}$ instead. The lemma below follows from Lemma \ref{McDuffTolman1} and Lemma \ref{relationMaslovclassandChernclass}. 

\begin{lem}\label{weight=maslov} Let $\gamma\in \P$ with Hamiltonian $H$. If $x\in L$ is a fixed point of the Hamiltonian circle action of $\gamma$, then 
$$\mu_v(\sigma_x)=w(x)\quad and \quad \tau(\sigma_x)=-H(x)$$
\end{lem}

\begin{proof}   Let $\sigma^2_x$ denote the class section in $P_{\gamma^2}$ corresponding to $\sigma_x$. Recall that, by definition, $\gamma^2$ has Hamiltonian $2H$.   Thus, by Lemma \ref{McDuffTolman1} we have 
$$c_v(\sigma^2_x)=w(x)\quad and \quad \tau_{\gamma^2}(\sigma^2_x)=-2H(x).$$ By  Lemma \ref{relationMaslovclassandChernclass} $$c_v(\sigma^2_x)=\mu_v(\sigma_x) \quad \text{and}\quad \tau(\sigma^2_x)=2\tau_{\gamma}(\sigma_x)$$ and  the two equalities follow.
\end{proof}

 Let $\gamma\in\P$, $\tau\in \mathcal{T}(\gamma)$ and $J_P\in\J(P,\om,\tau,\Omega_{c})$ constructed from an $S^1$--invariant $J\in \J(M,\om)$ via the doubling procedure, that is $J_P$ is the pull-back of an $S^1$--invariant almost complex structure of $P_{\gamma^2}$ under the embedding $\iota_2:P_{\gamma}\sra P_{\gamma^2}$.  Fix $B\in H_2^D(M,L)$, and consider the moduli space of $J_P$--pseudo-holomorphic disks with no marked  points
$$ \ov{\M}(P_{\gamma},\sigma_{max}+B, J_P).$$
This moduli space consists of stable maps representing section classes $\sigma$. Concretely, those stable maps consist of   only one $J_P$--holomorphic section component  that we call the \emph{root}, and the other components   are $J_P$--holomorphic disks contained in some fibers of $P_{\gamma}$   that we will call \emph{bubbles}.  If the root represents the section class   $\sigma'\in H_2(P_{\gamma},N)$ and the bubbles represent fiber classes $B_i\in H_2(M,L)$, $i\in A$, we further have that $$\sigma_{max}+B=\sigma'+\sum_{i\in A}B_i.$$

\begin{prop}\label{propositionprincipale}
If $B\neq 0$ and $\om(B)\leq0$ the moduli space $\ov{\M}(P_{\gamma},\sigma_{max}+B, J_P)$ is empty. Furthermore, if $B=0$ then $J_P$ is regular,  and  the moduli space  $\ov{\M}(P_{\gamma},\sigma_{max}, J_P)$ is compact and can be identified with $L\cap F_{max}$.
\end{prop}

\begin{proof} We begin to show the first assertion. It is sufficient to show that for a  $J_P$--holomorphic section  with boundary on $N$ representing a section class $\sigma=\sigma_{max}+B$  one has:
\begin{equation}\label{eq:minimalarea}\Omega_{c}(\sigma)\geq  \Omega_{c}(\sigma_{max}),
\end{equation}
with equality only if $B=0$.  Indeed, suppose \eqref{eq:minimalarea} holds and assume there is a  $J_P$--holomorphic section representing $\sigma_{max}+B$ with, $B\neq 0$ and  $\om(B)\leq 0$. Then 
$$\Omega_{c}(\sigma)=\Omega_{c}(\sigma_{max}+B)=\Omega_{c}(\sigma_{max})+\om(B)\leq \Omega_{c}(\sigma_{max}).$$
This is impossible by \eqref{eq:minimalarea} unless $\om(B)=0$. This latter condition implies that  $B=0$ which contradicts $B\neq 0$. Let us now prove $\eqref{eq:minimalarea}$.  

Fix a point $[z,p]\in P_{\gamma}$ and consider $w\in T_{[z,p]}P_{\gamma}$. Write $w=h+v$ where $h$ and $v$ are respectively the  horizontal and vertical parts of $w$.  Choose $c>0$ such that $c> H_{max}$. Then, 
 \begin{eqnarray} \Omega_{c}(w,J_P w) &=&  (\om-dH\wedge dx- dH\wedge dy -\frac{H}{\pi} dx\wedge dy+\frac{c}{\pi}dx\wedge dy)(v+h, Jv+J_0h)\nonumber\\
 &=&  \om_p(v,Jv)+\frac{(c-H(p))}{\pi}dx\wedge dy(h,J_0h)\nonumber \\
&\geq& \frac{c-H_{max}}{\pi}dx\wedge dy(h,J_0h).\label{minarea}
\end{eqnarray}
where the last inequality holds since $J$ is $\om$--compatible. Since $\frac{1}{\pi}dx\wedge dy$ evaluates to one on the disc of radius one,
$$\Omega_{c}(\sigma)\geq (c- H_{max})=\Omega_{c}(\sigma_{max})$$
for a  $J_P$--holomorphic section with boundary on $N$ representing a section class $\sigma=\sigma_{x}+B$, with $x$ some fixed point. Note that equality in \eqref{minarea} only occurs  when the vertical part of $w$ vanishes. Hence, equality holds only when $B=0$ and $x\in F_{max}\cap L$.

Next, we show that $\M(P_{\gamma},\sigma_{max}, J_P)$ is compact and coincides with $L\cap F_{max}$.  Consider a stable map representing $\sigma$. Such a stable map  consists of exactly one root $\sigma'$ and possibly many bubbles  representing classes $B_i$, $i\in A$, with positive $\om$--area. It follows that the only stable maps representing a class $\sigma$ such that $$\Omega_{c}(\sigma)\leq(c-H_{max})$$ are the constant sections $\sigma_x$, with $x\in F_{max}\cap L$, which proves the claim.

There remains to show that $J_P$ is regular for $\sigma_{max}$. Let $\B$ denote the set of smooth maps $u: (D^2,\partial D^2)\sra (P_{\gamma},N)$ representing the class $\sigma_{max}$. For $u\in \B$, set 
$$\E_u:=C^{\infty}(\Lambda^{0,1}_{J_P}(D^2, u^*TP_{\gamma}))\quad and \quad \E:=\bigsqcup_{u\in \B}\E_u.$$
We have to show that the linearization of $$\delbar_{J_P}:\B\sra\E,\quad u\mapsto du+J_P\circ du\circ j$$ is surjective at every $u\in \ov{\M}(P_{\gamma},\sigma_{max}, J_P)$ (at least between suitable completions of the source and the target). Upto completion, this linearization is given by 
$$L_{\delbar_{J_P},u}:C^{\infty}(u^*TP_{\gamma}, u^*TN)\sra C^{\infty}( \Lambda^{0,1}(D^2,u^*TP_{\gamma})).$$
Since we only consider sections one only needs to verify that 
$$L^v_{\delbar_{J_P},u}:C^{\infty}(u^*Vert, u^*(TN\cap Vert))\sra C^{\infty}( \Lambda^{0,1}(D^2,u^*Vert)).$$
We show that the partial indices of the holomorphic bundle pair $(u^*TP_{\gamma}, u^*TN)$ must be at least bigger than -1.  Then we conclude by applying the results of Oh in \cite{OhRiemann-Hilbert}. Let $u$ be a $J_P$--holomorphic section representing $\sigma_x$ with $x\in F_{max}$. 
Since $x$ is a fixed point,  $u^*Vert$ reduces to $T_xM\cong \C^n$,   $u^*(TN\cap Vert)$ reduces to $T_xL\cong \R^n$ and the restriction of $u$ to $S^1$ defines a loop of lagrangian subspaces in $ \C^n$. Moreover,  this loop is given by 
$$\left.TL\right|_{u(e^{i2\pi t})}\equiv d\gamma(t) T_xL\subset T_xM.$$
Since the action is semi-free, and since $J_P$ comes from an $S^1$--invariant almost complex structure  $J$ of $(M,\om)$,  $ d\gamma(t):\C^n\sra \C^n$  takes the following diagonal expression after an appropriate change of basis of $\C^n$:  
\begin{equation}
  \label{eq:1}
  d\gamma(t):=
 \left( \begin{array}{cccc}  e^{i \pi m_{1} t}  & 0    & \cdots &  0         \\
                                          0 & \ddots &  \ddots    &\vdots \\
                                          \vdots & \ddots  & \ddots & 0         \\
                                          0 &       \cdots       & 0   & e^{i \pi m_n t}  
  \end{array}\right)
\end{equation}
 where $m_1,...,m_n$ are the the weights of the action at $x$ and are given  by $m_1=...=m_l=0$, $l=\dim F_{max}$, and $m_{l+1}=...=m_{n}=-1$. By projecting on each factor of $\C^n$ the initial Riemann-Hilbert problem   splits to a direct sum of 1-dimensional Riemann-Hilbert problems of the form: $$\begin{cases}\delbar \xi_j(z)=0\,\, \text{on $D^2$} &\\ 
 \xi_j(e^{2\pi i t})\in \R\la e^{i \pi m_{j} t}\frac{\partial}{\partial x_j} \ra & 
\text{} \\
 \end{cases} $$
where $\xi_j$ denotes the projection of $\xi: D^2\sra \C^n$ to the $j^{th}$ factor, and where $x_j$ are the real coordinates in $\C^n$. In this situation the partial indices coincide with the weights (Maslov indices) of each summand. Here partial indices are all greater than -1, but Oh  \cite{OhRiemann-Hilbert} proved that regularity holds for holomorphic discs with partial indices greater than -1 which ends the proof.
\end{proof}

%
%
%

To end the proof of Theorem \ref{calculelementdeSeidel},
we show the vanishing of all the other terms provided $F_{max}$ is of codimension two. 

\begin{prop} If $\codim(F_{max})=2$, then  $a_B=0$ for all $B\in\pi_2(M,L)$ with $\mu_L(B) >0$.
\end{prop}
\begin{proof} This is done by a simple dimension argument. Note that for $\sigma=\sigma_{max}+B$, the moduli space $\M^{pearl}([L_-],x_+, [\sigma])$ is empty unless
$$|x_+|=n+\mu_v(\sigma_{max})+\mu_L(B).$$
 From Lemma \ref{weight=maslov} and equation \eqref{eq:1} we have 
$$n=\dim(F^L_{max})-w_{max}=\dim(F^L_{max})-\mu_v(\sigma_{max}).$$ 
This implies that
$$\mu_L(B)\leq\codim (F^L_{max}).$$
 In particular, it follows from monotonicity of $L$ that if $B$ is representable by a $J$--pseudo-holomorphic disk one must have  $$2\leq\codim (F^L_{max}).$$
We conclude that when $F_{max}$ is of codimension exactly 2, there are no contributions in the Seidel element coming from $\sigma_{max}+B$ with $\om(B)>0$. 
\end{proof}

\subsection{Proof of Theorem \ref{lengthminimizing}}
 
The idea here is to adapt M.Akveld and D.Salamon's line of proof for length minimizing exact Lagrangian loops in $\C P^n$ (see \cite{AkveldSalamon}). We will make use of the following general result they showed:

\begin{prop}[M.Akveld, D.Salamon, \cite{AkveldSalamon}, Lemma 5.2 and 5.3]\label{propositionAkveldSalamon} Let $\gamma\in\P_L\Ham(M,\om)$ with $\gamma\in\P$. Let $A\in H_2(P_{\gamma},N)$ be a section class. Suppose that for any $\tau\in \mathcal{T}^{\pm}(\gamma)$ there exists  a family $J=\{J_z\}_{z\in D^2}$ of $\om$--tame almost complex structures in $M$ such that the moduli space $\M(P_{\gamma},A;\tau, \pm J)$ is not empty. Then,
\begin{equation*}\epsilon^+(N)\geq -\la[\tau_0],A\ra\quad\text{and}\quad \epsilon^-(N)\leq -\la[\tau_0],A\ra
\end{equation*}
for any connection 2--form $\tau_0\in\mathcal{T}(\gamma)$. 
\end{prop} 

We begin by observing that in Proposition \ref{propositionprincipale}, the results are independant of the choice of connection 2--form in $\mathcal{T}(\gamma)$.  Note that the arguments in this proposition apply to show that $J_P$ is regular for $\sigma_{min}$, assuming the minimum fixed point set $F_{min}$ to be semi-free. Also, $\M(P_{\gamma},\sigma_{min},J_P)$ is non empty and coincides with $F_{\min}\cap L$.

 With this in mind, one argues as follows.   Assume $\tau\in \mathcal{T}^+(\gamma)$. Then, there is a regular $J_P$ such that $\M(P_{\gamma},\sigma_{max}, J_P)$ is non-empty. Similarly, assuming $\tau\in \mathcal{T}^-(\gamma)$,  there is a regular $J_P$ such that $\M(P_{\gamma},\sigma_{min}, J_P)$ is non-empty. Let $\tau_0=\tau_{\gamma}$. Then, by Proposition \ref{propositionAkveldSalamon} one has 
 \begin{eqnarray*}\epsilon(N)&=& \epsilon^+(\tau_{\gamma},N)- \epsilon^-(\tau_{\gamma},N)\\
&\geq& -\la[\tau_{\gamma}],\sigma_{max}\ra+\la[\tau_{\gamma}],\sigma_{min}\ra\\
 &\geq & \la[\om-H dx\wedge dy-dH\wedge dx-dH\wedge dy],\sigma_{min}-\sigma_{max} \ra\\
 &\geq & -(H_{min}-H_{max})\\
 &=& \ell(N)
 \end{eqnarray*}
It follows from  \cite[Theorem B]{AkveldSalamon} that:
$$ \ell(N)\leq\epsilon (N)\leq \nu(N).$$
 hence the proof.

\section{Application to Fano toric manifolds}

\subsection{Toric manifolds: the Delzant construction} The following is taken from \cite{CoxKatz} or \cite{AnnaCannasdaSilva} or \cite{GuilleminSternberg}. Let $\la,\ra:\R^{n}\times (\R^{n})^*\sra\R$ denote the standard pairing. Symplectic toric manifolds are compact connected symplectic manifolds $(M^{2n},\om)$ together with an effective Hamiltonian action of $\mathbb{T}^n$ and a choice of corresponding moment map $\mu$. It is well-known that the image $\Delta:=\mu(M)\subset (\R^n)^*$ is a  convex \emph{polytope}, meaning that this is an intersection of a collection of affine half planes in $(\R^{n})^*$.  Such half planes are determined by  vectors $\{v_i\}_{i\in 1,...,d}$ in $\R^n$ and real numbers $\{a_i\}_{i\in 1,...,d}$. Explicitly, the polytope is given by:
\begin{eqnarray*} \Delta:=\{f\in (\R^{n})^*| \la f,v_i\ra\geq a_i, i\in 1,...,d\}.
\end{eqnarray*}
The $v_i$'s represent inward-pointing normal vectors to the facets of the polytope, and the faces of $\Delta$ are in bijection with the sets
$$F_I:=\{f\in (\R^n)^*|\la f,v_i\ra=a_i, i\in I\},\quad I\subset[1,n],\quad F_I\neq\emptyset$$

Symplectic toric manifolds are in 1-1 correspondence with  \emph{Delzant polytopes}, i.e. polytopes  verifying:
\begin{itemize}
\item[1)] each vertex has $n$ edges.
\item[2)] the edges at any vertex $p$ are rational in the sense that they are given by some $p+tf_i$ with $t\in [0,1]$ and $f_i\in \Z^n$, $i=1,...,n$.
\item[3)] at each vertex the corresponding vectors $f_1,...,f_n$ can be chosen to be a $\Z$-basis of $\Z^n$.
\end{itemize}
The symplectic toric manifold $M$ with moment polytope $\Delta$ can be realized as a  symplectic reduction of a Hamiltonian torus action of $\mathbb{T}^{d-n}$ on $(\C^{d},\om_{st})$. The construction is as follows. Let $\{e_i\}_{i=1,...,d}$ denote the standard basis of $\R^d$. It is easy to see that the map  $\pi:\R^d\sra\R^n,\quad e_i\mapsto v_i$ descends to a surjective Lie group morphism:
$$\pi:\mathbb{T}^d\sra \mathbb{T}^n.$$
Let $N:=\ker\pi$. If $\iota:N\sra \mathbb{T}^d$ denotes the inclusion, then the composition of $\iota$ with the standard Hamiltonian action of $\mathbb{T}^d$ on $\C^d$
$$(e^{i\theta_1},...,e^{i\theta_d}).(z_1,...,z_d)=(e^{-2\pi i\theta_1}z_1,...,e^{-2\pi i\theta_d}z_d)$$  gives a Hamiltonian action of $\mathbb{T}^{d-n}$ on $\C^{d}$.  Let $\{w_1,...,w_{d-n}\}\in \ker \pi$  be a basis where $w_i=\sum_{j=1}^d w_i^j e_j$. 
  Then, 
$$\exp(w_i).(z_1,...,z_d)=(e^{-2\pi i w_i^1}z_1,...,e^{-2\pi iw_i^d}z_d)$$
Furthermore, considering  the following exact sequence of dualized Lie algebras:
$$0\sra (\R^n)^*\stackrel{\pi^*}{\sra}(\R^d)^*\stackrel{\iota^*}{\sra}(Lie(\ker\pi))^* \sra0,$$ 
and setting for $j=1,...,d$
$$\rho_j:=\iota^* e^*_j$$
then the action becomes:
$$\exp(w).(z_1,...,z_d)=(e^{-2\pi i \la\rho_1,w\ra}z_1,...,e^{-2\pi i\la\rho_d,w\ra}z_d).$$
 The moment of this action is then given by the composition $\iota^*\circ\mu_{st}$, where $$\mu_{st}(z_1,...,z_d)=(\pi|z_1|^2,..., \pi |z_d|^2)+(a_1,...,a_d).$$
 Explicitly one gets:
 \begin{eqnarray*}\iota^*\circ\mu_{st}(z_1,...,z_d)&=& \iota^*(\sum_{i=1}^d(\pi|z_i|^2+a_i)e_i^*)\\
 &=& \sum_{i=1}^d(\pi|z_i|^2+a_i)\rho_i\\
 &=& \sum_{i=1}^d \sum_{m=1}^{d-n}(\pi|z_i|^2+a_i)w_m^i w_m^*
 \end{eqnarray*}
 Then $0$ is a regular value for $\iota^*\circ\mu_{st}$. Moreover, $\ker\pi$ acts freely on the compact submanifold $Z:=(\iota^*\circ\mu_{st})^{-1}(0)$.  It follows that  $$M=\left.(\iota^*\circ\mu_{st})^{-1}(0)\right/N$$
is a compact manifold. Let  $\iota_Z:Z\sra \C^d$ denote the inclusion map  and  $p_M:Z\sra M$ denote the quotient map. Then, by the Marden-Weinstein theorem, $M$ is equipped with a canonical symplectic structure $\om$ such that:
$$p_M^*\om= \iota_Z^*\om_0.$$
With respect to $\om$ the action of the $n$--torus $\T^n=\left.\T^d\right/N$, which leaves $Z$ invariant,  is Hamiltonian. The corresponding moment map $\mu$ is defined by
$$\mu_{st}\circ\iota_Z=\pi^*\circ (\mu\circ p_M)$$  
 and  has  image  $\Delta$. \\

\subsection{Alternative construction of the toric manifold} Here we describe an alternative construction of the toric manifold $M$ as a complex manifold.  




Extend the map previously seen  $\pi:\R^d\sra\R^n$ to a mapping $$\pi_{\C}:\C^d\sra \C^n.$$
Note that  $\pi_{\C}$ sends the standard lattice $\Z^d$ to the set of primitive integral generators of the facets of $\Delta$. 
Hence it  induces a map  between complex tori $\pi_{\C}:\T^d_{\C}\sra \T^n_{\C}$. Let $N_{\C}$ denote the kernel of $\pi_{\C}$ so that we have an exact sequence of complex groups:
$$0\sra N_{\C}\sra\T^d_{\C}\sra \T^n_{\C}\sra 0. $$
Consider now the linear diagonal  action $\kappa$ of $\T_{\C}^d $ on $\C^d$ given by
$$\kappa(\exp w).(z_1,...,z_d)=(\exp(w_1)z_1,...,\exp(w_d)z_d).$$
For any subset   $I=\{i_1,...,i_k\}\subset \{1,...,d\}$ set
$$\C^d_I:= \{z\in\C^d| z_i=0\quad iff\quad i\in I\}.$$
Note that this set is a $\T^d_{C}$--orbit and every  $\T^d_{C}$--orbit is actually of this type. Now consider 
the following subspace of $\C^d$:
$$\C^d_{\Delta}:= \bigcup_{\{I|F_I\,\textrm{is a face of}\, \Delta\}}\C^d_I.$$
This open  subset of $\C^d$ is in fact the biggest subset on which $N_{\C}$ has no singular orbit: in fact $N_{\C}$ acts on $\C^d_{\Delta}$ freely and properly \cite{GuilleminSternberg}.

The corresponding quotient manifold  $\left.\C^d_{\Delta}\right/N_{\C}$ is a  compact and  complex. Moreover, the  $\T^d_{C}$--action on $\C^d_{\Delta}$ induces  $\T^n_{C}$--action on the quotient. 
This is this  quotient that  corresponds to $M$. The relation between the two constructions is expressed in the following theorem:

\begin{theorem} The manifold $Z$ is contained in $\C^d_{\Delta}$ and the restriction of an  $N_{\C}$--orbit to $Z$ is an $N$--orbit. 
\end{theorem}

\subsection{The anti-symplectic involution} This involution is the one induced by complex conjugation in $\C^d$. In fact, complex conjugation is well-defined on the subset $\C^d_{\Delta}$. Furthermore it commutes with the action of $N_{\C}$ since for every $w\in \T^d_{\C}$ we have
$$\ov{\kappa(w)(z)}=\kappa(\ov{w})\ov{z}.$$
Thus it defines an involution on the quotient space
$$\tau: M\sra M$$
satisfying 
\begin{equation}\label{actionVSinvolution} \tau(w\cdot x)= w^{-1}\cdot \tau (x),\quad \forall x\in M,\quad\forall w\in \T^d.
\end{equation}
That $\tau$ is anti-symplectic and that it preserves the moment map of the $\T^n$--action on $M$ then follows from the fact that complex conjugation is anti-symplectic with respect to the standard symplectic structure and that the  moment map associated to the diagonal $\T^d$--action on $\C^d$ is invariant under this conjugation. 

\subsection{The homology of toric manifolds and their real lagrangian.} 

In this section we describe the $\Z_2$--cohomology rings of the toric manifold $M$ and of its real lagrangian $L=\textrm{Fix}(\tau)$.  We will also explain how these rings are isomorphic, the isomorphism being given by a degree 2 ring homomorphism.

\subsubsection{On the homology of the  toric manifold}
The homology of the toric manifolds $M$ is generated by its \emph{toric divisors}, that is the  complex codimension 1 faces (facets) of $\Delta$. If  $D_1,...,D_d$ denote those facets then they are geometrically realized as follows 
$$D_k= Z\cap \C^d_k. $$
These determine codimension 2  cycles in $X$. Let $Y_1,...,Y_d\in H_{2n-2}(X)$ denote  the homology of these facets. Then 
$$H_*(M;\Z_2)=\frac{\Z_2[Y_1,...,Y_d]}{P(\Delta)+SR(\Delta)}$$
where $P(\Delta)$ and $SR(\Delta)$ denote the following ideals
$$P(\Delta):=\LA\sum_{k}\la \xi,v_k \ra Y_k| \xi\in (\Z_d)^* \RA$$
and
$$SR(\Delta):=\LA \prod_{i\in I}Y_i| \textrm{ $I\subset [1,d]$ is such that $D_I:=D_{i_1}\cap...\cap D_{i_k}= \emptyset$, $I$ is primitive} \RA$$
where $I$ is \emph{primitive} if for all $i_{m}\in I$, $D_{I\bs \{i_m\}}\neq \emptyset$.
In this setting the Chern class $c_1(X)$ of $X$ is given by  the Poincar\'e dual of  $Y_1+... +Y_d$.


We should also mention that there is a  natural isomorphism  between $H_2(M,\Z_2)$ and the set of tuples $A=(a_1,...,a_d)\in \Z^d$ such that 
$$\sum_{k}a_k v_k=0.$$
Under this isomorphism the pairing of $A$ with $PD(Y_i)$ (the Poincar\'e dual of $Y_i$) simply coincides with the projection to the $i$--factor of $A$:
$$\la A,PD(Y_i)\ra= a_i.$$
The following result due to Batyrev will be useful:

\begin{theorem}[Batyrev, \cite{Batyrev}]\label{Batyrev} For any primitive $I\in [1,d]$ there is a unique vector $a_I=(a_1,...,a_d)\in H_2(M,\Z_2)$ 
such  that:
 \begin{equation*} \textrm{$a_k=1$ for all $k\in I$},\quad \textrm{$a_k\leq 0$ for $k\notin I $}
 \end{equation*}
\end{theorem}

\subsubsection{The cohomology of the real lagrangian}

Let $g$ be a $\tau$--invariant Riemmanian metric on $M$. Let $g$ also  denote the restriction of $g$ to $L$. Note that for generic $\xi\in Lie(\T^n)$ the function
$$f_{\xi}: M\sra \R,\quad x\sra \la \mu(x),\xi\ra$$
is Morse. Moreover, there exists a second category Baire subset of $\tau$--invariant metrics such that the pair $(f_{\xi},g)$ is  Morse-Smale. Then  $Crit( f_{\xi})$ corresponds to the vertices of the Delzant polytope, i.e. the critical point of the moment map. Moreover, for any vertex $p$ the Morse index is given by $$|p|_M=2 \times \#\{\textrm{1--dimensional faces $\psi$ at $p$ such that $\la\psi, \xi\ra<0$}\},$$
hence $f_{\xi}$ is perfect. Let $\left.f_{\xi}\right|_L$ be the restriction of $f_{\xi}$ to $L$. It is not hard to see that 
$$Crit \left.f_{\xi}\right|_L =Crit  f_{\xi}. $$
Furthermore the restricted pair $(\left.f_{\xi}\right|_L,g)$  is also Morse-Smale. In fact,  the inclusion of $L$ into $M$ induces an isomorphism of Morse chain complexes:

\begin{theorem}[Duistermaat \cite{Duistermaat}, Haug \cite{Haug}]\label{theoremhomologielag} The map
\begin{equation}incl:Crit_k \left.f_{\xi}\right|_L\sra Crit_{2k} f_{\xi},\quad p\mapsto p\end{equation}
defines a ring isomorphism between Morse homologies with $\Z_2$--coefficients that doubles the degrees:
$$incl:H_*(L,\Z_2)\stackrel{\cong}{\longrightarrow} H_{2*}(M,\Z_2).$$ 
\end{theorem}

The restriction to $\Z_2$--coefficients is essential; it is needed to show that $ \left.f_{\xi}\right|_L$ is perfect (see \cite{Haug}). Note that it is easy to  find  examples where the theorem doesn't hold for other coefficient-rings, for example $\R P^n$ in $\C P^n$.

\subsubsection{The quantum cohomology of the real lagrangian} Set $$P_L(\Delta):=incl^{-1} P(\Delta)\quad\text{and}\quad SR_L(\Delta):= incl^{-1}(SR(\Delta)).$$
By theorem \ref{theoremhomologielag} and from the description of $H^*(M,\Z_2)$ we can write
$$H_*(L;\Z_2)=\frac{\Z_2[X_1,...,X_d]}{\la P_L(\Delta)+SR_L(\Delta)\ra}$$
where $X_j$ is a formal variable representing the homology class of $D_j\cap L$. 
Let $I=(i_1,...,i_d)$  be a multi-index of non negative integers. Then set:
$$X^I:= X_1^{i_1}...X_d^{i_d}.$$
The degree $|I|$ of $X^I$ is naturally $\sum_k i_kd_k$ where  $d_k$ stands for the degree of $X_k$. Since the homology of $L$ is generated by the classes of degree $n-1$, the Lagrangian is either $wide$ or $narrow$ according to \cite{BiranCorneaUniruling}. Luis Haug \cite{Haug} showed that the Floer differential vanishes for the standard   complex structure which happens to be generic, thus proving that real Lagrangians are actually wide:

\begin{theorem}\cite[Theorem A]{Haug} The real Lagrangians $L$ are wide as $d^Q$ generically vanishes. Furthermore, the isomorphism 
$QH_*(L;\Lambda_L)\cong H_*(L;\Z_2)\otimes \Lambda_L$
is canonical. 
\end{theorem}

By the theorem above each of the $X_i$  defines an element in $QH_*(L;\Lambda_L)$ also denoted $X_i$.  
By Theorem \ref{Batyrev}, for any primitive $I\subset [1,d]$  there is a unique vector $a_I=(a_1,...,a_d)\in H_2(M;\Z_2)$ 
such  that:
 \begin{equation}\label{relationBatyrev} \textrm{$a_k=1$ for all $k\in I$},\quad \textrm{$a_k\leq 0$ for $k\notin I $}
 \end{equation}
 From Theorem \ref{theoremhomologielag} the same relations exist in $H_*(L;\Z_2)$.     Write  $I=(i_1,..., i_l)$ and let $J=(j_1,..., j_m)$ denote the complement of $I$ in $[1,d]$. Set 
$$P_L^Q(\Delta):=P_L(\Delta)$$ 
and 
$$SR_L^Q(\Delta):=\LA X_{i_1}\cdots X_{i_l} -X^{|a_{j_1}|}_{j_1}\cdots X^{|a_{j_m}|}_{j_m}q^{-l+\sum^m_{r=1}|a_{j_r}|}| \textrm{ $I=(i_1,...,i_l)\subset [1,d]$, $I$ is primitive} \RA.$$ The remaining of this section is dedicated to show  the following: 

 
 \begin{prop}\label{Quantumpresentation}
$$QH(L;\Lambda_L)\cong \frac{\Z_2[X_1,...,X_d][q^{-1},q]}{P_L(\Delta)+SR_L^Q(\Delta)}.$$
\end{prop}

First, write
$$H_*(L;\Z_2)=\frac{\Z_2[X_1,...,X_d]}{\la f_1,...,f_r\ra}$$
where $f_1,...,f_k$ denote the polynomial relations in the $X_i$.
For $j=1,...,k$ let $f^Q_j$ to be the polynomials in the $X_i$ variables where the standard product is replaced by the quantum product.
The following two results are straighforward adaptations of lemmata  \cite[Lemmata]{SiebertTian}. We prove them for conveniency:

\begin{lem} The elements $X_1,...,X_d$ generate $QH_*(L;\Lambda_L)$. 
 \end{lem}

\proof Assume for simplicity that $X_1,...,X_d$ are critical points of a perfect Morse function on $L$. The proof is by decreasing induction on the degree of pure elements, starting at degree $deg=2n$. For $deg=2n$, since $N_L\geq 2$, the quantum differential $\left.d_Q\right|_{R\la Crit_{2n}(f) \ra}$ restricts to  $d_0$. Consequently, the unique maximum of $f$ defines an element $QH_{n}(L;\Lambda_L)$, namely $[L]$.

 Assume that pure elements of degree upto $deg$ are generated by $X_1,...,X_d$: we show that every monomial $X^I$ of degree $deg-1$ is generated by the  $X_k$. Let $X^{I,Q}$ denote the element obtained by making the Lagrangian quantum  product of $X_k$ with multi-index $I$. By definition of the quantum product we have that:
$$X^{I,Q}=X^I+ \sum_{j\geq 1, |R|\geq deg} \lambda_{R,j} X^{R}t^j,\quad \lambda_{R,j}\in \Z_2.$$
  It follows from the induction hypothesis that the $X^R$ in the equation above can be written as quantum products of the $X_k$; hence the conclusion.
\qed\\

Argumenting as in the preceding lemma, we obtain that:
\begin{eqnarray*}f^Q_j(X_1,...,X_d) & = & f_j(X_1,...,X_d)+ g^Q_j(X_1,..,X_d)\\
&=& g^Q_j(X_1,..,X_d)
\end{eqnarray*}
since $f_j$ is assumed to be a relation in homology. Thus, the polynomial in abstract variables $q_1,...,q_d$:
\begin{equation*}
f_j^{[\om]}(q_1,...,q_d):= f^Q_j(q_1,...,q_d)- g^Q_j(q_1,...,q_d)
\end{equation*}
define relations in the quantum homology when  we evaluate them at $(X_1,...,X_d)$.
 We have the following:
\begin{lem} \label{quantumhomologyrelations}
$$QH_*(L,\Lambda_L)=\left.\Z_2[q,q^{-1}] ([X_1,...,X_d])\right/ \la f^{[\om]}_1,...,f^{[\om]}_r\ra.$$
\end{lem}
 \proof We have seen that the $f_j^{[\om]}$ define relations in Quantum Homology. Let $\mathcal{I}$ denote the ideal generated by the $f_j^{[\om]}$, $j=1,...,r$. We show that any polynomial defining a relation in the quantum homology actually belongs to $\mathcal{I}$. Let $P^Q\in \mathcal{I}\bs \{0\}$ be of degree $deg$ (in the abstract variables $q_1,...,q_n$). Then, we can write $P^Q= P^Q_{deg}+R$ where $P^Q_{deg}$ is the degree $deg$ term and where $\deg R>deg$. Since $P^Q$ defines a relation, evaluating at $X_1,...,X_d $ gives:
 $$P^Q_{deg}(X_1,...,X_d)=-R(X_1,...,X_d).$$
 By definition of the quantum product, $P^Q_{deg}(X_1,...,X_d)$ can be written as a sum of a degree $deg$ polynomial  $P_{deg}(X_1,...,X_d)$ (where the product is the intersection product) and a polynomial of degree bigger  than $deg$. Thus,  $P_{deg}(X_1,...,X_d)=0$ which implies that there is a polynomial function $\phi$ such that $P_{deg}=\phi(f_1,...,f_r)$. Again, replacing the standard product by the quantum product gives:
 $$\phi(f_1^{\om},...,f_r^{\om})=P^Q_{deg}+R',\quad \deg R' >deg. $$
 This implies that $P^Q=\phi(f_1^{\om},...,f_r^{\om})+R-R'$, with $\deg (R-R')>deg$. To finish the proof we do a decreasing induction on the degree. 
\qed\\

\proof (Proposition \ref{Quantumpresentation})
Now, we  show how to use the formula for relative Seidel morphisms   in order to compute  $QH_*(L;\Lambda_L)$. For  a primitive $I=(i_1,...,i_l)\subset [1,d]$, let $J=(j_1,...,j_m)$ denote its complement in $[1,d]$. Consider the unique vector $a_I=(a_1,...,a_d)\in H_1(L;\Z_2)$ 
such  that:
 \begin{equation}\label{relationBatyrev2} \textrm{$a_k=1$ for all $k\in I$},\quad \textrm{$a_k\leq 0$ for $k\notin I $},\quad \sum a_kv_k=0.
 \end{equation}
  Let $\Lambda^{1/2}_j$ denote the half-turn map associated  to $\Lambda_j$ the $S^1$--action generated by the normal to the $j$--face of the Delzant polytope. Note that from \eqref{actionVSinvolution} we have  $\Lambda^{1/2}_j\in \P_L\Ham(M,\om)$.  Thus,   in terms of Hamiltonian paths preserving $L$, \eqref{relationBatyrev2} means that:
  \begin{equation}\label{relationchemin} (\Lambda^{1/2}_1)^{a_1}\cdots (\Lambda^{1/2}_d)^{a_d}=Id_L.
\end{equation}
Since we are dealing with a torus action, the order of the terms in the left member of \eqref{relationchemin} can be reorganized in such way that we finally get:
  \begin{equation}\label{relationchemin2} (\Lambda^{1/2}_{i_1})\cdots (\Lambda^{1/2}_{i_l})=  (\Lambda^{1/2}_{j_1})^{-a_{j_1}}\cdots (\Lambda^{1/2}_{j_m})^{-a_{j_l}}=(\Lambda^{1/2}_{j_1})^{|a_{j_1}|}\cdots (\Lambda^{1/2}_{j_l})^{|a_{j_m}|}.
\end{equation}
Observe that the maximum fixed point set of each $\Lambda_j$ is given by the divisor $D_j$, hence is of codimension 2, and both assumptions \emph{A1} and \emph{A2} are verified in this context. It follows from  Theorem \ref{calculelementdeSeidel} that $S_L(\Lambda^{1/2}_{j})=X_j\otimes q$. Thus,  considering the Lagrangian Seidel element associated to both sides in \eqref{relationchemin2} we have:
\begin{equation*}X_{i_1}\star\cdots \star X_{i_l}\otimes q^l=X_{j_1}^{|a_{j_1}|}\star\cdots\star X_{j_m}^{|a_{j_m}|}\otimes q^{\sum_{r=1}^m |a_{j_r}|}
\end{equation*}
where $\star$ stands for the Quantum Lagrangian product (see \cite{BiranCorneaQuantumHomology}).  As a consequence of Lemma \ref{quantumhomologyrelations},  these are the only multiplicative relations in $QH(L;\Lambda_L)$. Furthermore, the additive relations are the same as in standard homology. The presentation of $QH(L;\Lambda_L)$ follows.
\qed\\
 
 \begin{rem} \begin{itemize}
 \item[$\bullet$]
 This presentation is the same as given in \cite[Proposition 5.2]{McDuffTolman} for the toric manifold $(M,\om)$. Let 
 $\Lambda_M:=\Z_2[Q^{-1},Q]$ with $|Q|=2$, and  denote by $QH(M;\Lambda_M)$ the corresponding quantum homology. Then, as above, 
 $$QH(M;\Lambda_M)\cong\frac{\Z_2[Y_1,...,Y_n][Q^{-1},Q]}{P(\Delta)+SR^Q(\Delta)},$$
 where $Y_i$ represents the class of the divisor $D_i$, $P(\Delta)$ stands for the set of  linear relations between the divisors, and $SR^Q(\Delta)$ stands for the set of quantum multiplicative relations between the divisors.  
It follows directly that there is a  ring isomorphism:
 $$\psi: QH_*(L;\Lambda_L)\sra QH_{2*}(M;\Lambda_M)$$ 
 such that $\psi(X_i)=Y_i$ and $\psi(q)=Q$. This ring isomorphism is actually induced by the inclusion $incl$ in \ref{theoremhomologielag} as was shown by L.Haug in \cite{Haug}. 
 It is worth noticing that, in the notations above,  $\psi(S_L(\Lambda^{1/2}_j))=S(\Lambda_j)$, i.e. relative Seidel elements associated to loops dual to facets are sent to the corresponding absolute  Seidel elements under $\psi$. 
 \item[$\bullet$] Using the $QH(M;\Lambda_M)$-module structure of the quantum homology of $L$ one also has \cite{BiranCorneaUniruling, HuLalonde}
 $$S_L(\Lambda_j^{1/2})\star S_L(\Lambda_j^{1/2})= S_L((\Lambda_j^{1/2})^2)=S (\Lambda_j)\odot [L],$$
in other words  $S_L(\Lambda_j^{1/2})$ is somewhat a square root of $S (\Lambda_j)$.  As it was pointed out to me by Fran\c{c}ois Charette,   these relations completely determine the $QH(M;\Lambda_M)$-module structure of $QH(L;\Lambda_L)$. Moreover, it is possible to recover the $H(M)$--module structure of $H(L)$. For instance,  when $N_L\geq 3$ the quantum product of $X_i$ with itself does not admit any quantum correction term. Hence,
\begin{equation*} S_L(\Lambda_j^{1/2})\star S_L(\Lambda_j^{1/2})=(X_i\otimes q)\star (X_i\otimes q) = (X_i\star X_i) q^2= X_i^2 q^2
\end{equation*}
and since $$S (\Lambda_j)\odot [L]=Y_iQ\odot [L]=(Y_i\odot[L])Q=(Y_i\odot[L])q^2$$ we get $X_i^2 =Y_i\odot[L]$. 
\end{itemize}
  \end{rem}

 We end this Section with an example illustrating the results established so far.

\begin{ex}\label{Theexample}
Consider the pair $(\C P^3, \R P^3)$, where $\C P^3$ is equipped with the Fubini-Study form $\om_{FS}$. Let $(M,L)=(\wt{\C P^3}, \wt{\R P^3})$ denote the monotone Lagrangian blow-up of $(\C P^3, \R P^3)$  with symplectic form $\wt{\om}$ (see \cite{Rieser} for the definition). Topologically 
$$\wt{\C P^3}\cong\C P^3 \#\ov{\C P}^3\qquad \text{and}\qquad \wt{\R P^3}\cong\R P^3 \#\R P^3.$$
One can also view $\wt{\C P^3}$ as the  the projectivisation of the rank 2 complex bundle $\mathcal{O}(-1)\oplus \C \sra \Sigma$ where $\Sigma\cong \C P^2$ denotes the exceptional divisor. In this point of view, $\wt{\R P^3}$ is a non trivial $S^1$--bundle over $\R P^2$. The group $H_2(M;\Z_2)$ is generated by the class $F$ of the fiber of this fibration and the class $E$ of the exceptional curve ( $E=L-F$ where $L=[\C P^1]$ is the class of a line.). A simple computation shows that $H^D_2(M, L;\Z_2)$ is  generated by  half of   $E$ and half of $F$. We will use the same notations to refer to them.
 
Now, the  symplectic form for the blow-up of weight $\lambda$ is  given by $[\wt{\om}]=[\phi^*\om_{FS}]-\pi\lambda^2 e$, where   $\phi : \wt{\C P^3}\sra \C P^3 $ is the blowing-down map, and where  $e\in H^2(\wt{C P}^3,\Z)$ is the  Poincar\'e dual of $\Sigma$.
  Monotonicity then forces  $\lambda$ to be $\sqrt{2}/2$.  The torus $\T^3$ acts in a Hamiltonian way on $\C P^3$ as follows
$$(\theta_1,\theta_2,\theta_3)\cdot [z_0:z_1:z_2:z_3]=[z_0:e^{-2\pi i\theta_1}z_1:e^{-2\pi i\theta_2}z_2:e^{-2\pi i\theta_3}z_3]$$ The moment map of this action is 
$$\mu([z_0:z_1:z_2:z_3])=\left(\frac{\pi |z_1|^2}{\sum_{i=0}^3|z_i|^2},\frac{\pi |z_2|^2}{\sum_{i=0}^3|z_i|^2},\frac{\pi |z_3|^2}{\sum_{i=0}^3|z_i|^2}\right)$$
 so that the moment polytope is given by
$$\Delta=\{(x_1,x_2,x_3)\in \R^3|0\leq x_i, \,\,x_1+x_2+x_3\leq \pi\}.$$
This action lifts to a Hamiltonian action of $\T^3$ on the blow-up with moment map $\wt{\mu}$. The corresponding moment polytope   can be identified with
$$\wt{\Delta}=\{(x_1,x_2,x_3)\in \R^3|0\leq x_1,\,\,0\leq x_2,\,\, 0\leq x_3\leq \pi/2, \,\,x_1+x_2+x_3\leq \pi\}.$$
Moreover, the restriction of $\wt{\mu}$ to $\wt{\R P}^3$ has also image $\wt{\Delta}$. Now, the outward normals to the facets are 
$$v_1=(-1,0,0),\,\,v_2=(0,-1,0),\,\, v_3=(0,0,-1),\,\, v_4=(0,0,1),\,\, v_5=(1,1,1).$$
We introduce some more notations. Let  $\Lambda_i$, $i=1,...,5$ denote the semi-free  Hamiltonian circle action  fixing the  facets defined by  the $v_i$'s. Let $\Lambda_i^{1/2}$ denote the Hamiltonian path corresponding to half of $\Lambda_i$ and let $X_i$ be formal variables representing the intersection of the divisors associated to the facets with normals $v_i$ with $L$.  To compute the Quantum Homology of $L$ first note that 
$$P(\wt{\Delta})=\la X_1=X_2=X_5,\,\, X_3=X_4+X_5 \ra,\qquad  SR(\wt{\Delta})=\la X_1X_2X_5=0,\,\, X_4(X_4+X_5)=0 \ra.  $$
Setting $X=X_1$ and $Y=X_4$ and applying Theorem \ref{Quantumpresentation} yields:
 $$QH(L;\Lambda_L)=\Z_2[X,Y, q^{\pm1}]/\la X^3=Y q^{-2}, Y(X+Y)=[L]\otimes q^{-2} \ra$$
 which is indeed isomorphic to $QH(M;\Lambda_M)$. 
 
 It is not hard to see that the product $YX$ has no quantum term. Set $YX=\partial  E$, then $Y\star Y=\partial E+[L]\otimes q^{-2}$. Thus, $$S_L((\Lambda_4^{1/2})^2)=(S_L(\Lambda_4^{1/2}))^2=(Y\star Y)\otimes q^2=E\otimes q^2+[L]$$
 and we see a lower order term appearing. Note however that the action of   $(\Lambda_4)^2$ is not semi-free on the maximum subset. 
 
  Finally, we wish to show that lower order terms may appear when the maximum fixed point set is of codimension strictly greater than 2 in $M$.   Consider the circle  action $\Lambda$ associated to the combination $v_1+v_2+v_4$. The maximum fixed point set is semifree and corresponds to the  point mapped to the intersection $D_1\cap D_2\cap D_4$ under $\wt{\mu}$. Then, we have
\begin{equation*}
S_L(\Lambda)=S_L(\Lambda_1^{1/2})\star S_L(\Lambda_2^{1/2})\star S_L(\Lambda_4^{1/2})= (X\star X\star Y)\otimes q^3.
\end{equation*}
It is not hard to check that $X\star X$ coincides with the intersection product $X\cdot X=\partial F$. In order to compute the Lagrangian  quantum product $(X\cdot X)\star Y $, observe that 
$$(X\cdot X)\star Y =[pt]+\alpha Y\otimes q^{-2}+\beta X\otimes q^{-2},\qquad \alpha,\beta\in \{0,1\}$$
for dimensional reasons and since $F$ and $E$ are the only effective  Maslov 2 classes. By a direct computation one has $\alpha=1$, hence the conclusion. 
\end{ex}

 \subsection{Lagrangian uniruledness}

 Recall that $H_n(L)\otimes\Lambda_L$ embeds in  $QH(L;\Lambda_L)$ canonically.  We set $Q_-$ to be the complement:
 $$Q_-=H_*(\oplus_{k<n} Crit_k(f)\otimes \Lambda_L, d^Q).$$ 
 
 \begin{lem} \label{unirulingandinvertibles} Let $L\subset M$ be a closed monotone Lagrangian with $N_L\geq 2$. Assume that $L$ is not narrow.  If $L$ is not Lagrangian uniruled all the invertible elements of $QH(L,\Lambda_L)$ can be written as 
 $$\lambda[L]+ x,\quad\text{where $\lambda\in \Lambda_L\bs\{0\}$ et $x\in Q_-$}.$$
 \end{lem}
 
 \proof Assuming $L$ is not uniruled we show that $Q_-$ is an ideal in $QH(L,\Lambda_L)$. Suppose it is not the case, then there would $x\in QH(L,\Lambda_L)$ and $y\in Q_-$ such that $x\star y$ has a term of the type $r[L] t^{\mu(B)}$ with
 $$r=\#_2\{\M^{pearl}(x,y;[L], B)\}\neq 0.$$ Note that $B\neq 0$ since in that case $r=0$ unless both $x$ and $y$ have index $n$. It follows by assumption that   $Q_-$ is an ideal.  Since the unit $[L]$ cannot belong to $Q_-$ unless $QH(L;\Lambda_L)$ vanishes, any invertible must have such a presentation.
 \qed\\
 
Now we are prepared to prove Theorem \ref{Lagrangianuniruledness}.
 
 \begin{proof}[Theorem \ref{Lagrangianuniruledness}]
 Assume $L$ is not narrow, then the claim follows from Lemma \ref{unirulingandinvertibles}. If $L$ is narrow, then the fundamental class $[L]$ is a $d^Q$--boundary. Since $[L]$ is represented by the unique maximum of some generic Morse function on $L$, this implies that there is a pseudo-holomorphic disk through the maximum, which ends the proof.
 \end{proof}
 
 Here is another consequence of the Lemma above
 
 \begin{cor} Any real  monotone Lagrangian with $N_L\geq 2$ in a Toric manifold is Lagrangian uniruled.
   \end{cor}

\proof It  was seen that the  fundamental class of $[L]$ do not appear in any Lagrangian Seidel element associated to an $S^1$--circle action fixing one of the codimension one faces of the Delzant polytope. However, Lagrangian Seidel elements are all invertibles. The claim then follows from Lemma \ref{unirulingandinvertibles}.
\qed\\

%

\bibliographystyle{plain}
\bibliography{bibarticle}

\end{document}